\documentclass[11pt]{amsart}

\usepackage[a4paper,margin=1in]{geometry}
\usepackage{amsmath,amssymb,amsthm,mathrsfs}
\usepackage{hyperref}
\usepackage{enumitem}
\usepackage{mathtools}
\usepackage{bm}
\usepackage{tikz}
\usepackage{comment}

\hypersetup{
	colorlinks=true,
	linkcolor=blue,
	citecolor=blue,
	urlcolor=blue
}

\newtheorem{theorem}{Theorem}[section]
\newtheorem{lemma}[theorem]{Lemma}
\newtheorem{proposition}[theorem]{Proposition}

\theoremstyle{remark}
\newtheorem{remark}[theorem]{Remark}

\newcommand{\mn}{\medskip\noindent}

\numberwithin{equation}{section}

\title{
	A Local--to--Global Propagation Principle
	for Dirichlet--to--Neumann Maps
}

\author[T. Daud\'e]{Thierry Daud\'e}

\address{
	Universit\'e Marie et Louis Pasteur,
	CNRS, LmB (UMR 6623),
	F-25000 Besan\c{c}on, France
}

\email{thierry.daude@univ-fcomte.fr}

\author[A. Enciso]{Alberto Enciso}

\address{
	Instituto de Ciencias Matem\'aticas,
	Consejo Superior de Investigaciones Cient\'ificas,
	C/ Nicol\'as Cabrera 13--15,
	28049 Madrid, Spain
}

\email{aenciso@icmat.es}

\author[B. Helffer]{Bernard Helffer}

\address{
	Laboratoire de Math\'ematiques Jean Leray,
	Nantes Universit\'e,
	2 rue de la Houssini\`ere,
	BP 92208,
	F-44322 Nantes Cedex 03, France
}

\email{Bernard.Helffer@univ-nantes.fr}

\author[N. Kamran]{Niky Kamran}

\address{
	Department of Mathematics and Statistics,
	McGill University,
	805 Sherbrooke Street West,
	Montr\'eal, QC, H3A 0B9, Canada
}

\email{niky.kamran@mcgill.ca}

\author[F. Nicoleau]{Fran\c{c}ois Nicoleau}

\address{
	Laboratoire de Math\'ematiques Jean Leray,
	Nantes Universit\'e,
	2 rue de la Houssini\`ere,
	BP 92208,
	F-44322 Nantes Cedex 03, France
}

\email{francois.nicoleau@univ-nantes.fr}

\date{}

\begin{document}
	
\maketitle

\begin{abstract}
	
	We establish four local-to-global propagation results for
	Dirichlet--to--Neumann maps. Our first two results are proved in the general setting of smooth
	compact Riemannian manifolds with boundary. The first shows that if two
	smooth Riemannian metrics coincide in a collar neighborhood of a
	connected boundary component \(\Gamma\), then equality of the
	corresponding local Dirichlet--to--Neumann maps on a nonempty open
	subset of \(\Gamma\) propagates to equality of the associated global
	Dirichlet--to--Neumann maps on all of \(\Gamma\). The proof combines
	unique continuation and self-adjointness arguments. The second replaces
	the geometric collar assumption by an exponential spectral assumption
	on the difference of the corresponding global
	Dirichlet--to--Neumann maps. The proof relies on the spectral unique
	continuation theory of Jerison--Lebeau, through the formulation of
	Le~Rousseau--Lebeau. Our third and fourth results establish local-to-global propagation
	principles under Ingham-type quasi--analytic spectral assumptions.
	Assuming that the boundary manifold is respectively a compact
	Riemannian symmetric space or a compact quasi--analytic Riemannian
	manifold, they rely on the propagation theorems of Ganguly--Thangavelu and of Bhowmik--Pradhan. As an application, we consider a class of conformally warped product
	metrics. In this setting, the local Borg--Marchenko theorem and
	Weyl--Titchmarsh theory relate the required Ingham-type spectral decay
	to a suitable quasi--analytic boundary closeness of the conformal
	factors, yielding new local-to-global uniqueness results for
	Dirichlet--to--Neumann maps.

\end{abstract}

	\tableofcontents


\section{Introduction}\label{Intro}

Since Calder\'on's pioneering work \cite{Calderon}, inverse boundary value problems have generated an extensive literature. One of the central questions is whether the coefficients of an elliptic equation, or the underlying geometry of a manifold, can be recovered from boundary measurements encoded by the Dirichlet--to--Neumann (DN) map.

\medskip\noindent
The case of partial boundary measurements has attracted particular attention during the last two decades. For the conductivity and Schrödinger equations in Euclidean domains, important uniqueness results with partial data were obtained by Bukhgeim--Uhlmann \cite{BukhgeimUhlmann}, Kenig--Sjöstrand--Uhlmann \cite{KSU07}, Isakov \cite{Isakov}, Imanuvilov--Uhlmann--Yamamoto \cite{ImanuvilovUhlmannYamamoto}, and very recently Uhlmann--Wang \cite{UhlmannWang}. Similar questions have also been investigated on Riemannian manifolds, notably by Kenig--Salo \cite{KenigSalo}. We also refer to the surveys \cite{KenigSaloSurvey,UhlmannSurvey} for a broader overview of the subject.

\medskip\noindent
The philosophy of the partial data Calder\'on problem is that measurements performed on a suitable open subset of the boundary should determine global information inside the manifold. In the present paper, we investigate a different but related question.
Instead of asking whether local measurements determine the metric, we ask
whether equality of local Dirichlet--to--Neumann maps forces equality of
the corresponding global Dirichlet--to--Neumann maps, thereby reducing the
inverse problem to a setting where stronger uniqueness results are
available.

\medskip\noindent
To investigate this issue, we first introduce the global Dirichlet--to--Neumann map associated with a smooth metric \(g\).
Let \(M\) be a smooth compact orientable connected manifold with smooth boundary.  Let \(u\) be the solution of
\begin{equation}
	\label{eq:laplace-dirichlet}
	-\Delta_g u=0
	\quad \text{on } M\,,
\end{equation}
with boundary condition
\[
u|_{\partial M}=\psi\,,
\qquad
\psi\in C^\infty(\partial M)\,.
\]
The corresponding global Dirichlet--to--Neumann map is defined by
\begin{equation}
	\label{eq:global-DN-map}
	\Lambda_g\psi
	=
	\partial_\nu u\big|_{\partial M}\,,
\end{equation}
where \(\nu\) denotes the outward unit normal vector field along \(\partial M\).

\medskip\noindent
Given a nonempty open subset
\(
\mathcal O\subset \partial M\,,
\)
we define the local Dirichlet--to--Neumann map as follows.
For \(\psi\in C_c^\infty(\mathcal O)\), let \(u\) be the solution of
\eqref{eq:laplace-dirichlet} with boundary condition
\[
u=\psi
\quad\text{on }\mathcal O\,,
\qquad
u=0
\quad\text{on }\partial M\setminus\mathcal O\,.
\]
The local Dirichlet--to--Neumann map is then given by
\begin{equation}
	\label{eq:local-DN-map}
	\Lambda_{g,\mathcal O}\,\psi
	=
	\partial_\nu u\big|_{\mathcal O}\,.
\end{equation}

\medskip\noindent
We first state a geometric local-to-global result in which the boundary may have several connected components. The point is that the propagation takes place along one fixed connected boundary component: equality of the local Dirichlet--to--Neumann maps on an arbitrary nonempty open subset of this component extends to the whole component, provided the two metrics agree in a collar neighborhood of it.

\begin{theorem}[Geometric local-to-global propagation]
	\label{firstmainresult}
	Let \(g\) and \(\widetilde g\) be smooth Riemannian metrics on \(M\).
	Assume that
	$
	g=\widetilde g
	$
	in a collar neighborhood of the connected boundary component \(\Gamma\). If the corresponding local Dirichlet--to--Neumann maps coincide on a nonempty open subset
	$
	\mathcal O\subset \Gamma\,,
	$
	then
	$
	\Lambda_{g,\Gamma}
	=
	\Lambda_{\widetilde g,\Gamma}
	$
	on the whole boundary component \(\Gamma\)\,.
\end{theorem}

\medskip\noindent
The proof combines boundary unique continuation with a self-adjointness argument. While the result may be well known to specialists, we have not been able to locate an explicit reference. 
We include it here because it isolates a simple local-to-global propagation mechanism based on self-adjointness, which will reappear in a different context in the proof of our main result.

\medskip\noindent
As we shall see in the final part of the introduction (see (\ref{Borg})), in the warped product
setting, the local Borg--Marchenko theorem shows that coincidence of the
metrics in a collar neighborhood of the boundary is equivalent to a
fixed exponential decay of the spectral coefficients of the difference
of the corresponding Dirichlet--to--Neumann maps. This naturally leads to the following abstract
propagation result, in which the geometric collar assumption of
Theorem~\ref{firstmainresult} is replaced by an exponential spectral assumption.

\begin{theorem}[Exponential spectral propagation] \label{expthm}
Let $M$ be a smooth compact orientable connected manifold with smooth
connected boundary $K=\partial M$, and let $\mathcal O\subset K$ be a nonempty open
subset. Let $g$ and $\widetilde g$ be two smooth Riemannian metrics on $M$ inducing
the same Riemannian metric $g_K$ on $K$. Let
\[
0=\lambda_0<\lambda_1<\lambda_2<\cdots
\]
be the distinct eigenvalues of the Laplace--Beltrami operator $\Delta_{g_K}$ on $K$, and let
\[
P_k:L^2(K)\longrightarrow E_{\lambda_k}
\]
be the associated spectral projections. Assume that
$
\Lambda_{g,\mathcal O}
=
\Lambda_{\widetilde g,\mathcal O}
$
and that there exist constants $C>0$ and $\varepsilon>0$ such that
\[
\|P_k (\Lambda_g-\Lambda_{\widetilde g})\|_{\mathcal L(L^2(K))}
\le
C e^{-\varepsilon\sqrt{\lambda_k}}\,,
\qquad k\ge0\,.
\]
Then
$
\Lambda_g=\Lambda_{\widetilde g}\,.
$
\end{theorem}

\medskip
\begin{remark}
	Set
	$
	A=\Lambda_g-\Lambda_{\widetilde g}\,.
	$
	Since the two metrics induce the same boundary metric on \(K\)\,, the two
	Dirichlet--to--Neumann maps have the same principal symbol. Hence
	$
	A\in\Psi^0(K)\,,
	$
	that is, \(A\) is a pseudodifferential operator of order \(0\). In
	particular, \(A\) extends to a bounded operator on \(L^2(K)\). The theorem extends without change to manifolds with nonconnected
	boundary by considering each connected boundary component separately.
\end{remark}

\mn
When the boundary \(K\) is a compact real-analytic Riemannian manifold,
a theorem of Seeley \cite{Seeley1969} characterizes real analyticity by
the exponential decay of the eigenfunction coefficients. First observe that our assumption
on the spectral projections immediately implies 
that, with $f=(\Lambda_g-\Lambda_{\tilde g})\psi$, we have
\begin{equation}\label{Pkf}
	\|P_kf\|_{L^2(K)}
	\le
	Ce^{-\varepsilon\sqrt{\lambda_k}}\,,
	\qquad k\ge0\,.
\end{equation}
In particular,  if $(\phi_\ell)$ is an orthonormal basis of eigenvectors of the Laplace Beltrami operator $-\Delta_{g_K}$, we get
$P_k \phi_\ell=\phi_\ell$ when $\lambda_k$ is the eigenvalue 
 associated with $\phi_\ell$ ($k=k(\ell)$) and 
  we obtain
 \begin{equation}\label{coef-seeley}
 	\big|\langle f,\phi_\ell\rangle\big|
 	\le
 	Ce^{-\varepsilon\sqrt{\lambda_{k(\ell)}}}\,,
 \end{equation}
 which corresponds to Seeley's criterion on the individual eigenfunction coefficients.

\par\noindent
Conversely, Seeley's
hypothesis implies the exponential decay of the spectral projections up
to the factor \(\sqrt{d_k}\), where
\(d_k=\dim E_{\lambda_k}=O(\lambda_k^{(n-1)/2})\) by Weyl's 
law, with optimal remainder \cite{Hormander68}.
Hence this polynomial factor can be absorbed into the exponential decay
by replacing  \(\epsilon\) with any constant \(\epsilon'<\epsilon\). 

\par\noindent
Thus, in the real-analytic setting, Theorem~\ref{expthm} can be proved
using the spectral characterization of analyticity together with the
self-adjointness of \(\Lambda_g-\Lambda_{\widetilde g}\)\,, without
appealing to the Jerison--Lebeau propagation theorem.

\mn 
We emphasize that the proof of Theorem~\ref{expthm} does not require any analyticity
assumption on the boundary manifold \(K\)\,. The quasi-analytic
propagation results established below are obtained under additional
assumptions on \(K\)\,, namely that it is either a compact Riemannian
symmetric space or a compact quasi-analytic Riemannian manifold.
Their strength is that they replace the exponential spectral decay in
Theorem~\ref{expthm} by the much weaker Ingham-type quasi-analytic
decay~\cite{Ingham1934}. It would be interesting to determine whether these additional
assumptions on \(K\) are genuinely necessary, or whether the same
quasi-analytic propagation principle remains valid for arbitrary smooth
compact Riemannian manifolds, as is the case for the exponential
propagation theorem of Jerison--Lebeau.

\mn 
We now assume that the boundary manifold \(K\) is a compact (necessarily orientable) connected
Riemannian symmetric space. Typical examples include flat tori,
spheres, odd-dimensional real projective spaces, complex projective
spaces, compact Lie groups, and finite products of such spaces.  In this setting, Ganguly and Thangavelu \cite{GT} established an
Ingham-type quasi-analytic propagation theorem, building on the
classical work of Ingham \cite{Ingham1934}; see also Koosis
\cite{Koosis} for an elegant exposition of Ingham's theory. This leads
to the following local-to-global propagation result.

\begin{theorem}[Ingham-type propagation on symmetric spaces]
	\label{GTthm}
	Let $M$ be a smooth compact orientable connected manifold with smooth
	connected boundary $K=\partial M$, and let
	$\mathcal O\subset K$ be a nonempty open subset.
	Let $g$ and $\widetilde g$ be two smooth Riemannian metrics on $M$
	inducing the same Riemannian metric $g_K$ on $K$. Assume that $(K,g_K)$ is a compact connected Riemannian symmetric space. Let
	\[
	0=\lambda_0<\lambda_1<\lambda_2<\cdots
	\]
	be the distinct eigenvalues of the Laplace--Beltrami operator
	$\Delta_{g_K}$ on $K$\,, and let
	\[
	P_k:L^2(K)\longrightarrow E_{\lambda_k}
	\]
	be the associated spectral projections. Assume that
	$
	\Lambda_{g,\mathcal O}
	=
	\Lambda_{\widetilde g,\mathcal O}\,,
	$
	and that there exist a constant $C>0$ and a positive decreasing
	function
	\[
	\theta(t)\longrightarrow0
	\quad\text{as }t\to+\infty\,,
	\qquad
	\int_T^{+\infty}\frac{\theta(t)}{t}\,dt=+\infty\,,
	\]
	and
	\begin{equation}\label{estspectral}
	\|P_k(\Lambda_g-\Lambda_{\widetilde g})\|_{\mathcal L(L^2(K))}
	\le
	C\exp\!\left(-\sqrt{\lambda_k}\,
	\theta(\sqrt{\lambda_k})\right)\,,
	\qquad k\ge0\,.
	\end{equation}
	Then
$
	\Lambda_g=\Lambda_{\widetilde g}\,.
$
\end{theorem}

\mn 
\begin{remark}
	The conclusion of Theorem~\ref{GTthm} remains valid if the global
	spectral estimate (\ref{estspectral}) is replaced by the following local pointwise
	assumption: for every \(\psi\in C_c^\infty(\mathcal O)\) and every
	\(\omega\in\mathcal O\), there exists a constant
	\(C_{\omega,\psi}>0\) such that
	\begin{equation}
	|P_k (\Lambda_g-\Lambda_{\widetilde g} )\psi(\omega)|
	\le
	C_{\omega,\psi}\,
	e^{-\sqrt{\lambda_k}\theta(\sqrt{\lambda_k})}\,,
	\qquad k\ge0\,.
	\end{equation}
Indeed, the proof is unchanged: one simply uses the pointwise version,
Theorem~1.4 of Ganguly and Thangavelu~\cite{GT}, instead of the
\(L^2\)-version, Theorem~1.3. It would be interesting to investigate whether such pointwise estimates
can be derived directly from local geometric assumptions. More
precisely, if the two metrics coincide in a neighborhood of
\(\mathcal O\)\,, then the operator
$
A=\Lambda_g-\Lambda_{\widetilde g}
$
should possess strong microlocal regularity properties over
\(\mathcal O\). In the \(C^\infty\) category, this would suggest rapid decay of the
corresponding spectral components. It would be interesting to
investigate whether additional real-analytic regularity assumptions
could lead to exponential pointwise spectral decay estimates, thereby
providing a genuinely local route to the pointwise hypothesis above.
\end{remark}

\mn 
We next consider the case where the boundary manifold \(K\) is a
compact quasi-analytic Riemannian manifold. We briefly recall the
necessary background on quasi-analytic manifolds and Denjoy--Carleman
classes in the proof of the following theorem. Our result relies on the
recent propagation theorem of Bhowmik and Pradhan~\cite{BP}. Following
their work, we restrict ourselves to the family of Ingham weights
\begin{equation}
\theta_\delta(t)=\frac{1}{(\log t)^\delta}\,,
\qquad 0<\delta\le1\,.
\end{equation}

\begin{theorem}[Ingham-type propagation on quasi-analytic manifolds]
	\label{BPthm}
	Let $M$ be a smooth compact orientable connected manifold with smooth
	connected boundary $K=\partial M$, and let
	$\mathcal O\subset K$ be a nonempty open subset.
	Let $g$ and $\widetilde g$ be two smooth Riemannian metrics on $M$
	inducing the same Riemannian metric $g_K$ on $K$. Assume in addition that $(K,g_K)$ is a compact connected quasi-analytic
	Riemannian manifold. Let
	\[
	0=\lambda_0<\lambda_1<\lambda_2<\cdots
	\]
	be the distinct eigenvalues of the Laplace--Beltrami operator
	$\Delta_{g_K}$ on $K$, and let
	\[
	P_k:L^2(K)\longrightarrow E_{\lambda_k}
	\]
	be the associated spectral projections. Assume that
	$
	\Lambda_{g,\mathcal O}
	=
	\Lambda_{\widetilde g,\mathcal O}\,,
	$
	and that there exist constants $C>0$ and
	\(0<\delta\le1\) such that
	\[
	\|P_k(\Lambda_g-\Lambda_{\widetilde g})\|_{\mathcal L(L^2(K))}
	\le
	C
	\exp\!\left(
	-\sqrt{\lambda_k}\,
	\theta_\delta(\sqrt{\lambda_k})
	\right),
	\qquad k\ge0\,.
	\]
	Then
	$
	\Lambda_g=\Lambda_{\widetilde g}\,.
	$
\end{theorem}


\mn
We now turn to an application of the preceding propagation
principles to a class of conformally warped product metrics. Our
starting point is the geometric assumption in
Theorem~\ref{firstmainresult}. Although natural from the viewpoint of
boundary unique continuation, one may wonder whether the requirement
that the two metrics coincide in a collar neighborhood of the boundary
is really necessary, or whether it can be replaced by a weaker
boundary closeness condition. In the warped product setting, the answer is provided by
Weyl--Titchmarsh theory. Using the diagonalization of the
Dirichlet--to--Neumann map together with the
Weyl--Titchmarsh representation established in~\cite{DKN}, we derive
quantitative estimates for the spectral coefficients of the difference
of two Dirichlet--to--Neumann maps from suitable boundary closeness
conditions on the conformal factors. Combined with the abstract
propagation principles established in the first part of the paper,
these estimates yield new local-to-global uniqueness results in the
warped product setting.

\mn
More precisely, we consider non-compact manifolds of the form
\begin{equation}\label{warped-manifold}
	M=(0,1]\times K\,,
\end{equation}
where \(r\in(0,1]\) is a radial variable and \(K\) is either a compact
orientable connected Riemannian symmetric space or a compact orientable connected
quasi-analytic Riemannian manifold. The boundary of \(M\) is naturally
identified with the boundary manifold
\[
\partial M=\{1\}\times K\simeq K\,.
\]

\par\noindent
We endow \(M\) with the warped product metric
\begin{equation}\label{warped-metric}
	g=c(r)^4\bigl(dr^2+r^2g_K\bigr)\,,
\end{equation}
A natural question then is whether the manifold structure of $M$ and the Riemannian metric $g$ can be extended across $r=0$ to provide a smooth compact Riemannian manifold $(\check M, \check g)$ with connected boundary $\partial \check M = \{1\}\times K$ (which we identify with $K$). As we now briefly recall, the answer to this question is generally no, so that we need to allow for a geometric setting that is possibly singular at $r=0$. First, at the level of the manifold structure, the extension of the differentiable structure of the non-compact manifold $M=(0,1]\times K$ across $r=0$ as a compact orientable smooth manifold $\check M$ with connected boundary $\partial {\check M}=K$, in other words the "fillability" of K as the boundary of a compact orientable smooth manifold, puts significant topological conditions on $K$, such as the vanishing of all the Stiefel-Whitney numbers of its tangent bundle (these conditions are in fact necessary and sufficient if $\dim K \not\equiv 4 \mod 1,2,3$; the case $\dim K \equiv 4$ requires in addition the vanishing of the Pontrjagin numbers of the tangent bundle, see~\cite{MS}). These imply for example that the Euler characteristic
\(\chi(K)\) should be even. In particular, among compact
Riemannian symmetric spaces, they rule out examples such as the
complex projective plane \(\mathbb{P}_{2}(\mathbb{C})\), since
\(\chi(\mathbb{P}_{2}(\mathbb{C}))=3\). This is why the range of $r$ in our model is restricted to the semi-open interval $(0,1]$. Second, assuming that the topological conditions guaranteeing  the "fillability" of K as the boundary of a compact orientable smooth manifold are satisfied, the additional requirement that the metric \eqref{warped-metric} extend smoothly across the point corresponding to $r=0$ imposes rather restrictive conditions. For example, when $K=S^{d-1}$, smoothness of the metric at the origin requires that the transverse metric $g_K$ be the standard round metric on the sphere and that all odd-order derivatives of the warping factor vanish at $r=0$, namely
\[
c^{(2k+1)}(0)=0,
\qquad k\ge 0,
\]
see \cite[Section~4.3.4]{Pe}. In general, neither of these assumptions will be imposed here.
Consequently, the metric \eqref{warped-metric} should be regarded as a possibly
singular warped product metric near \(r=0\), so that our model generally
gives rise to a metric cone with a singular vertex at \(r=0\), except in
very special cases. This explains why the geometric applications below
are formulated in this singular setting, rather than for a smooth compact
completion.

\medskip\noindent
We thus consider metrics of the form
\[
g=c(r)^4(dr^2+r^2g_K)\,,
\qquad
\widetilde g=\widetilde c(r)^4(dr^2+r^2g_K)\,,
\]
where \(g_K\) is a Riemannian metric on a closed orientable manifold \(K\)\,. The warped product structure allows one to separate variables and to
diagonalize the Dirichlet--to--Neumann maps in the eigenbasis of the
Laplace--Beltrami operator on \(K\)\,. In particular, if
$
c(1)=\widetilde c(1)\,,
$
then, as explained above, the difference
$
A=\Lambda_g-\Lambda_{\widetilde g}
$
is a bounded operator on \(L^2(K)\) and is diagonal with respect to the
spectral decomposition of the Laplace--Beltrami operator on \(K\)\,, namely
\begin{equation}\label{decompfourier}
	A=\sum_{k\ge0}a_kP_k\,.
\end{equation}
A careful reading of the proof of the local Borg--Marchenko theorem
(\cite[Theorem~1.4]{DKN}) shows that the exponential decay
\begin{equation}\label{Borg}
	|a_k|
	\le
	Ce^{-\epsilon\sqrt{\lambda_k}}\,,
\end{equation}
is equivalent to the coincidence of the conformal factors in a collar
neighborhood of the boundary. Consequently, in this particular setting,
the exponential spectral assumption of the previous theorem is equivalent
to the collar assumption of Theorem~\ref{firstmainresult}. 

\mn 
Our main objective is to identify geometric boundary closeness
conditions that give rise to the quasi-analytic spectral estimates
appearing in Theorems~\ref{GTthm} and~\ref{BPthm}.
More precisely, we seek conditions on the conformal factors that imply
the decay estimate
\[
|a_k|
\le
Ce^{-\sqrt{\lambda_k}\theta(\sqrt{\lambda_k})}\,,
\]
where \(\theta\) is an admissible Ingham weight in the symmetric-space
setting, and
\[
\theta(t)=\theta_\delta(t)=\frac1{(\log t)^\delta}\,,
\qquad 0<\delta\le1\,,
\]
in the quasi-analytic setting.
The local-to-global propagation results then follow immediately from
Theorems~\ref{GTthm} and~\ref{BPthm}. To formulate the corresponding geometric
application, it is convenient to introduce the logarithmic variable
\[
x=-\log r,
\qquad r=e^{-x}\,,
\]
which identifies \(M\) with the infinite cylinder
\(
[0,\infty)\times K.
\)
In these coordinates, the metrics take the form
\begin{equation}
	g=f(x)^4(dx^2+g_K)\,,
	\qquad
	\widetilde g=\widetilde f(x)^4(dx^2+g_K)\,,
\end{equation}
where
\begin{equation}
	f(x)=e^{-x/2}c(e^{-x})\,,
	\qquad
	\widetilde f(x)=e^{-x/2}\widetilde c(e^{-x})\,.
\end{equation}

\medskip\noindent
We first consider the symmetric-space setting, so that
Theorem~\ref{GTthm} applies.  
The key point is to relate the geometric boundary closeness of the
conformal factors to the spectral assumption in
Theorem~\ref{GTthm}. This is achieved by combining the local
Borg--Marchenko theorem with quantitative Weyl--Titchmarsh estimates.
The abstract propagation result of Theorem~\ref{GTthm} then immediately
yields the following local-to-global uniqueness theorem.

\begin{theorem}[Symmetric-space case]\label{applicationsymetrique}
	Let \((K,g_K)\) be a compact orientable connected Riemannian symmetric space, and let
	\[
	g=c(r)^4(dr^2+r^2g_K)\,,
	\qquad
	\widetilde g=\widetilde c(r)^4(dr^2+r^2g_K)
	\]
	be two warped product metrics on
	\(M=(0,1]\times K\)\,, where
	\(c,\widetilde c\in C^m([0,1])\)\,, \(m\ge2\)\,, are positive functions. Define
	\[
	f(x)=e^{-x/2}c(e^{-x})\,,
	\qquad
	\widetilde f(x)=e^{-x/2}\widetilde c(e^{-x})\,.
	\]
	Assume that there exist \(C>0\) and \(\varepsilon>0\) such that
	\[
	|f^{(j)}(x)-\widetilde f^{(j)}(x)|
	\le
	Ce^{\Phi(x)}\,,
	\qquad
	j=0,1,2,
	\quad
	x\in(0,\varepsilon],
	\]
	where
	\[
	\Phi(x)
	=
	\inf_{t\ge T}\bigl(2xt-t\theta(t)\bigr)\,,
	\]
	and \(\theta:[T,\infty)\to(0,\infty)\) is a positive decreasing function
	satisfying
	\[
	\theta(t)\longrightarrow0\,,
	\qquad
	\int_T^\infty\frac{\theta(t)}{t}\,dt=+\infty\,.
	\]
	Let \(\mathcal O\subset K\) be a nonempty open subset. If
	$
	\Lambda_{g,\mathcal O}
	=
	\Lambda_{\widetilde g,\mathcal O}\,,
	$
	then
	$
	\Lambda_g
	=
	\Lambda_{\widetilde g}\,.
	$
\end{theorem}

\medskip\noindent
The assumptions of Theorem~\ref{applicationsymetrique} deserve some
further comments. The function $\Phi(x)$ plays a central role in the proof. It is introduced so that the
weighted Laplace transform
\begin{equation}\label{weighted-laplace}
	I(\rho)
	=
	\int_0^\varepsilon
	e^{-2\rho x}e^{\Phi(x)}\,dx
\end{equation}
satisfies the estimate
\begin{equation}\label{weighted-laplace-estimate}
	I(\rho)
	\le
	e^{-\rho\theta(\rho)}\,,
	\qquad
	\rho\ge T\,,
\end{equation}
(see Lemma~\ref{LaplaceEstimate}). Furthermore, the assumptions on
\(\theta\) imply that
\begin{equation}\label{Phi-minus-infinity}
	\Phi(x)\longrightarrow-\infty\,,
	\qquad
	x\to0^+,
\end{equation}
(see Lemma~\ref{explosion}). Thus the condition
\begin{equation}
	|f^{(j)}(x)-\widetilde f^{(j)}(x)|
	\le
	Ce^{\Phi(x)}\,,
	\qquad j=0,1,2,
\end{equation}
imposes a quantitative boundary closeness condition on \(f\) and
\(\widetilde f\), whose strength depends on the choice of \(\theta\). A typical example is obtained by choosing
\begin{equation}
	\label{eq:theta-log}
	\theta(t)=\frac1{\log t}\,.
\end{equation}
In this case, a simple minimization yields
\begin{equation}
	\label{eq:Phi-example}
	\Phi(x)
	\sim
	-\frac{4}{e}\,x^2 e^{\frac1{2x}},
	\qquad x\to0^+\,.
\end{equation}
Since \(x=-\log r\sim 1-r\) as \(r\to1^-\)\,, this corresponds to a boundary
closeness condition of the form
\begin{equation}
	\label{eq:boundary-closeness-example}
	|c(r)-\widetilde c(r)|
	\le
	C\exp\!\left(
	-\,\frac{4}{e}(1-r)^2
	e^{\frac{1}{2(1-r)}}
	\right),
	\qquad r\to1^-.
\end{equation}
If \(c\) and \(\widetilde c\) are \(C^\infty\) up to the boundary, such a
condition implies that their difference is flat at \(r=1\). In particular,
if the conformal factors belong to a quasi-analytic class, for instance if
they are real-analytic on \([0,1]\), then this already implies
$
c=\widetilde c
$
on \([0,1]\), and no further argument is required. However, the boundary estimate
\eqref{eq:boundary-closeness-example} should not be confused with a
quasi-analytic regularity assumption. Indeed, if \(\chi\in
C_c^\infty([0,1])\) satisfies \(\chi\equiv1\) in a neighborhood of
\(r=1\), then replacing \(c-\widetilde c\) by
\(\chi (c- \widetilde c)\) leaves the estimate
\eqref{eq:boundary-closeness-example} unchanged. The resulting functions vanish identically on a nonempty interior region and therefore cannot, in general, belong to a quasi--analytic class unless they vanish identically.

\par\noindent
A noteworthy feature of the present theorem is therefore that no
quasi-analytic regularity is assumed on the conformal factors. The
assumptions are compatible with conformal factors of finite regularity,
for instance \(C^m\), \(m\ge2\), where neither flatness propagation nor
quasi-analytic continuation can be invoked.

\medskip\noindent
Finally, the example \eqref{eq:theta-log} may be viewed as a borderline admissible
choice since
\begin{equation}
	\int_T^\infty \frac{dt}{t\log t}=+\infty\,,
	\qquad
	\int_T^\infty \frac{dt}{t(\log t)^{1+\varepsilon}}<\infty\,,
	\quad \varepsilon>0\,.
\end{equation}
Hence \(\theta(t)=1/\log t\) is, up to slowly varying factors, the fastest
decay compatible with the quasi--analyticity condition.

\begin{remark}
	While the conclusion of the theorem is an equality of global Dirichlet--to--Neumann maps, the global uniqueness result of \cite{DKN} then yields
	$
	g=\widetilde g.
	$
	However, the novelty of the theorem is not this uniqueness statement itself, but rather the propagation mechanism showing that equality of the local Dirichlet--to--Neumann maps on an arbitrary nonempty open subset \(\Gamma\subset K\) forces equality of the corresponding global Dirichlet--to--Neumann maps.
\end{remark}

\medskip\noindent
A distinctive feature of Theorem \ref{applicationsymetrique} is that the local-to-global
propagation does not stem from the usual unique continuation machinery
for partial differential equations. Instead, it is achieved through a
purely spectral argument combining Weyl--Titchmarsh asymptotics with a
quasi-analytic propagation theorem of Ingham's type. To the
best of our knowledge, this mechanism is new in the context of
Calder\'on-type inverse problems. 

\mn
The previous theorem relies on the fact that the transversal manifold
\(K\) is a compact orientable Riemannian symmetric space, and hence possesses an
intrinsic real-analytic structure.
Using a recent result of Bhowmik and Pradhan \cite{BP}, we show that the
same conclusion remains valid under the much weaker assumption that
\(K\) is merely a compact orientable quasi--analytic Riemannian manifold. 

\mn
Our next theorem provides the corresponding quasi--analytic analogue of
Theorem~\ref{applicationsymetrique}. For technical reasons, and in order to keep
the presentation as simple as possible, we restrict ourselves to the
particular class of weights
\begin{equation}
	\theta(t)=\frac{1}{(\log t)^\delta}\,,
	\qquad
	0<\delta\le 1\,.
\end{equation}

\begin{theorem}[Quasi-analytic case]\label{applicationquasianalytic}
	Let \((K,g_K)\) be a compact connected quasi--analytic Riemannian manifold,
	and let \(c,\widetilde c\in C^m([0,1])\)\,, \(m\ge2\), be positive
	functions. Consider the metrics
	\[
	g=c(r)^4(dr^2+r^2g_K)\,,
	\qquad
	\widetilde g=\widetilde c(r)^4(dr^2+r^2g_K)
	\]
	on \(M=(0,1]\times K\). Define
	\[
	f(x)=e^{-x/2}c(e^{-x}),
	\qquad
	\widetilde f(x)=e^{-x/2}\widetilde c(e^{-x})\,.
	\]
	For \(0<\delta\le1\), we introduce
	\[
	\Phi_\delta(x)
	=
	\inf_{t\ge T}
	\left(
	2xt-\frac{t}{(\log t)^\delta}
	\right)\,.
	\]
	Assume that there exist \(C>0\) and \(\varepsilon>0\) such that
	\[
	|f^{(j)}(x)-\widetilde f^{(j)}(x)|
	\le
	Ce^{\Phi_\delta(x)}\,,
	\qquad
	j=0,1,2,\quad x\in(0,\varepsilon]\,.
	\]
	Let \(\mathcal O\subset K\) be a nonempty open subset. If
	$
	\Lambda_{g,\mathcal O}
	=
	\Lambda_{\widetilde g,\mathcal O}\,,
	$
	then
	$
	\Lambda_g
	=
	\Lambda_{\widetilde g}\,.
	$
\end{theorem}

\begin{remark}
	A straightforward minimization shows that
	\[
	\Phi_\delta(x)
	\sim
	-\delta(2x)^{1+1/\delta}
	\exp\!\left((2x)^{-1/\delta}\right)\,,
	\qquad x\to0^+\,.
	\]
	Consequently, the assumptions of
	Theorem~\ref{applicationquasianalytic} require the conformal factors
	to be extremely close near the boundary \(r=1\), although this
	closeness is substantially weaker than coincidence in a collar
	neighborhood.
\end{remark}

\mn 
Beyond the warped product setting, the present work suggests that
Weyl--Titchmarsh theory may serve as a bridge between geometric
boundary closeness conditions and abstract quasi-analytic propagation
principles. It would be interesting to investigate whether similar
mechanisms exist in more general geometric settings.

\medskip\noindent
The paper is organized as follows. Section~2 is devoted to the proof of
the geometric local-to-global propagation theorem. In Section~3, we
prove the exponential spectral propagation theorem using the spectral
unique continuation theorem of Jerison--Lebeau in the formulation of
Le~Rousseau--Lebeau. Section~4 establishes the quasi--analytic
propagation principle underlying the proofs of
Theorems~\ref{GTthm} and~\ref{BPthm}. The remaining sections introduce
the warped product framework together with the associated
Weyl--Titchmarsh theory, establish the required quantitative spectral
decay estimates, and combine them with the abstract propagation results
to obtain the corresponding local-to-global uniqueness theorems.


\section{A geometric propagation argument}

In this section, we prove Theorem~\ref{firstmainresult}. We retain the
notation and assumptions introduced there. The geometric hypothesis that
the two metrics coincide in a collar neighborhood of the boundary
component \(\Gamma\) is illustrated in Figure~\ref{fig:collar}.

\begin{figure}[h]
	\centering
	\includegraphics[width=0.55\textwidth]{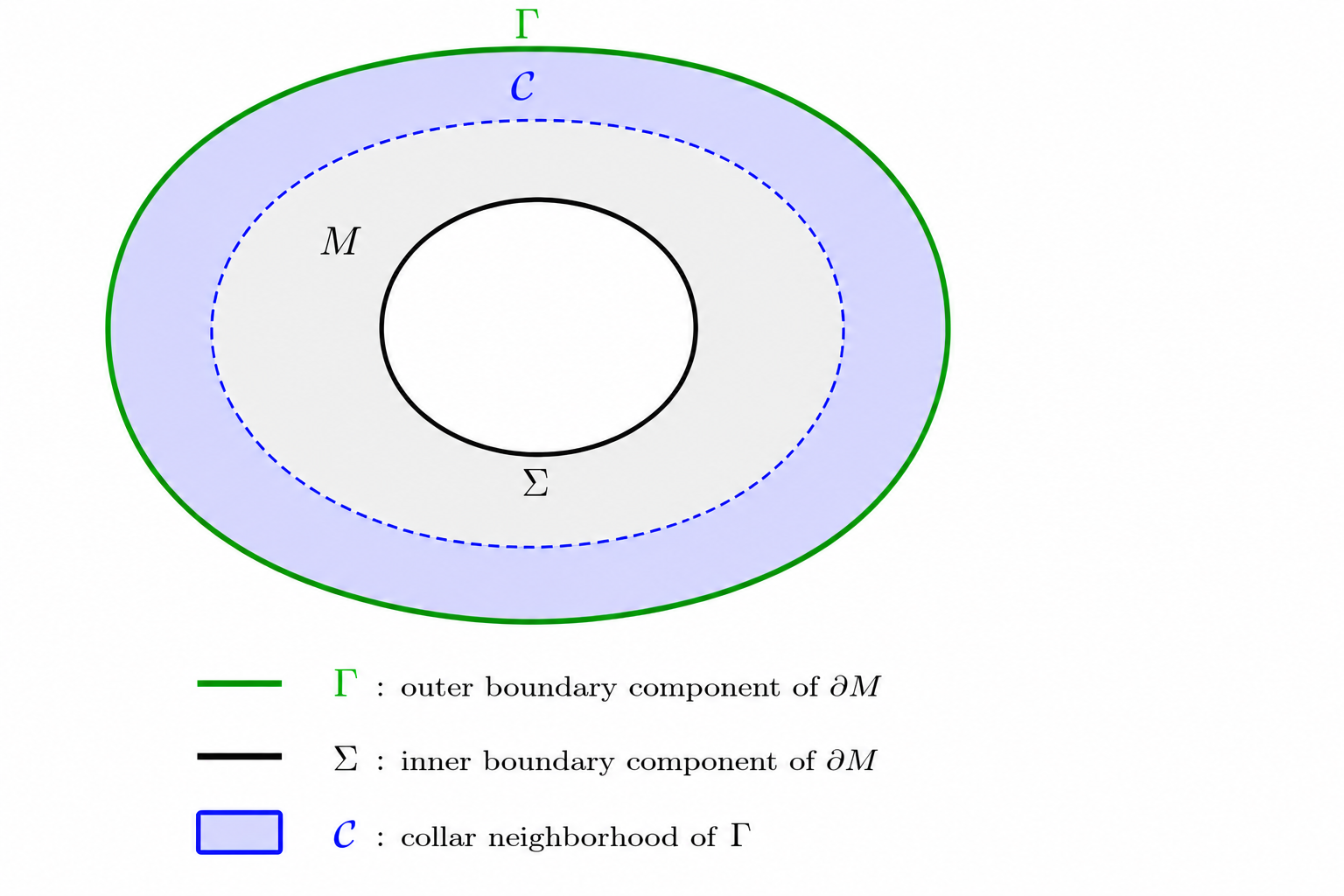}
	\caption{A collar neighborhood \(\mathcal C\) of the boundary component \(\Gamma\).}
	\label{fig:collar}
\end{figure}

\mn
\begin{proof}
Since \(g=\widetilde g\) in the collar neighborhood \(\mathcal C\), we denote their common value by
\[
g_0:=g=\widetilde g
\text{ in }\mathcal C\,.
\]
Consequently, \(g\) and \(\widetilde g\) induce the same boundary metric, boundary measure, and outward unit normal vector field on \(\Gamma\). We denote by
\(
\langle\cdot,\cdot\rangle
\)
the associated \(L^2(\Gamma)\) inner product. Define
\begin{equation}
	\label{eq:A}
	A:=\Lambda_{g,\Gamma}-\Lambda_{\widetilde g,\Gamma}\,.
\end{equation}
Since both Dirichlet--to--Neumann maps are self-adjoint with respect to
\(
\langle\cdot,\cdot\rangle\,,
\)
the operator \(A\) is self-adjoint on \(L^2(\Gamma)\)\,.
Let
\(
f\in C_c^\infty(\mathcal{O})\,.
\)
Let \(u\) and \(\widetilde u\) solve
\begin{equation}
	\label{eq:Dirichlet-problem}
	\begin{cases}
		-\Delta_g u=0 & \text{ in } M\,,\\
		u=f & \text{on } \Gamma\,,\\
		u=0 & \text{on } \partial M\setminus \Gamma\,,
	\end{cases}
	\qquad
	\begin{cases}
		-\Delta_{\widetilde g}\widetilde u=0 & \text{ in } M\,,\\
		\widetilde u=f & \text{on } \Gamma\,,\\
		\widetilde u=0 & \text{on } \partial M\setminus \Gamma\,.
	\end{cases}
\end{equation}
Define
$
w:=u-\widetilde u\,.
$
Since \(u\) and \(\widetilde u\) are harmonic with respect to the same metric \(g_0\) in the collar neighborhood \(\mathcal C\), it follows that
\begin{equation}
	\label{eq:w-harmonic}
	-\Delta_{g_0}w=0
	\text{ in }\mathcal C\,.
\end{equation}
Moreover, $ w|_\Gamma=0$. The local DN equality gives
\begin{equation}
	\label{eq:normal-derivative}
	\partial_\nu w=0
	 \text{ on } \mathcal{O}\,.
\end{equation}
By boundary unique continuation for elliptic equations (see \cite{Ho}, Section~28 or  \cite{SaloUCP}), it follows that
$ 	w\equiv0$ in $\mathcal C$. In particular,
\begin{equation}
	\label{eq:Af-global}
	Af=\partial_\nu w|_\Gamma=0
	 \text{ on } \Gamma\,.
\end{equation}
We have therefore shown that
\begin{equation}
	\label{eq:Af-omega}
	Af=0
	 \text{ on } \Gamma\,,
	\qquad f\in C_c^\infty(\mathcal{O})\,.
\end{equation}
We now use self-adjointness. Let \(F\in C^\infty(\Gamma)\). For every
\(h\in C_c^\infty(\mathcal{O})\), we have
\begin{equation}
	\label{eq:selfadjointness}
	\langle AF,h\rangle
	=
	\langle F,Ah\rangle
	=
	0\,.
\end{equation}
Hence \(AF=0\) on \(\mathcal O\). Repeating the preceding boundary
unique-continuation argument with boundary datum \(F\), we obtain
\(AF=0\) on \(\Gamma\). Since \(F\in C^\infty(\Gamma)\) was arbitrary,
this proves \(A=0\), that is, $\Lambda_{g,\Gamma}
=
\Lambda_{\widetilde g,\Gamma}$\,.

\end{proof}


\section{An exponential spectral propagation argument}

In this section, we prove Theorem~\ref{expthm}. The geometric collar
assumption is now replaced by an exponential decay condition on the
spectral components of the difference of the Dirichlet--to--Neumann maps.
We shall use the exponential spectral unique continuation theorem of
Jerison--Lebeau~\cite{JL}. We also refer to the interpolation
inequality of Lebeau--Zuazua~\cite[(5.6)]{LZ}; see also
Lebeau--Robbiano~\cite{LR} and
Le~Rousseau--Lebeau~\cite[Proposition~5.6]{LL}. We recall that
\begin{equation}
\label{Lap-spectrum}
0=\lambda_0<\lambda_1<\lambda_2<\cdots
\end{equation}
are the distinct eigenvalues of the Laplace--Beltrami operator on \(K\)\,, and
\begin{equation}
\label{spectral-proj}
P_k:L^2(K)\longrightarrow E_{\lambda_k}
\end{equation}
is the associated spectral projection. Assume that
\[
g=\widetilde g
\text{ on }K=\partial M\,,
\]
and set
\begin{equation}
\label{A-def}
A=\Lambda_g-\Lambda_{\widetilde g}\,.
\end{equation}
Since
$
\Lambda_{g,\mathcal O}
=
\Lambda_{\widetilde g,\mathcal O}\,,
$
we have, for every
\(
\psi\in C_c^\infty(\mathcal O),
\)
\begin{equation}
\label{Apsi-local}
A\psi=0
\text{ on }\mathcal O\,.
\end{equation}
Assume moreover that there exist constants \(C>0\) and
\(\varepsilon>0\) such that
\begin{equation}
\label{exp-decay}
\|P_kA\|_{\mathcal L(L^2(K))}
\le
Ce^{-\varepsilon\sqrt{\lambda_k}},
\qquad k\ge0\,.
\end{equation}
Then
\begin{equation}
\label{exp-decay-Apsi}
\|P_k(A\psi)\|_{L^2(K)}
\le
Ce^{-\varepsilon\sqrt{\lambda_k}}
\|\psi\|_{L^2(K)}\,,
\end{equation}
so that the spectral components of \(A\psi\) satisfy a fixed
exponential decay. Since \(A\psi=0\) on \(\mathcal O\)\,, all the
assumptions of Proposition~5.6 in
Le~Rousseau--Lebeau~\cite{LL} are fulfilled. Therefore,
\begin{equation}
\label{Apsi-zero}
A\psi=0
\text{ on }K\,.
\end{equation}
Now, since \(A\) is self-adjoint, the same argument as in the proof of
Theorem~\ref{firstmainresult} yields
\[
\langle Au,\psi\rangle
=
\langle u,A\psi\rangle
=
0\,,
\]
for every \(u\in C^\infty(K)\) and every
\(\psi\in C_c^\infty(\mathcal O)\)\,. Hence
\[
Au=0
\text{ on }\mathcal O.
\]
Moreover, by \eqref{exp-decay},
\begin{equation}
\label{exp-decay-Au}
\|P_k(Au)\|_{L^2(K)}
\le
Ce^{-\varepsilon\sqrt{\lambda_k}}
\|u\|_{L^2(K)}\,.
\end{equation}
Proposition~5.6 of \cite{LL} applies once again and yields
\begin{equation}
\label{Au-zero}
Au=0
\text{ on }K\,.
\end{equation}
Since this holds for every
\(
u\in C^\infty(K)\,,
\)
we conclude that
$
\Lambda_g=\Lambda_{\widetilde g}\,.
$
This completes the proof.


\section{The quasi--analytic propagation principle}

In this section, we prove Theorems~\ref{GTthm} and
\ref{BPthm}. The common idea is to replace the exponential spectral
propagation used in Theorem~\ref{expthm} by an Ingham-type
quasi--analytic propagation theorem. We first treat the case where the boundary \(K\) is a compact
Riemannian symmetric space, relying on the theorem of
Ganguly--Thangavelu. We then turn to compact quasi--analytic
Riemannian manifolds, using the recent propagation result of
Bhowmik--Pradhan.

\subsection{Propagation on symmetric spaces}

In this subsection, we assume that \(K\) is a compact connected
Riemannian symmetric space. Typical examples include
\begin{equation}
K=\mathbb S^d\,,\qquad
K=\mathbb T^d\,,\qquad
K=\mathbb{RP}^d\,,
\end{equation}
more generally compact rank-one symmetric spaces and finite products of
such spaces. We recall that
\[
0=\lambda_0<\lambda_1<\lambda_2<\cdots
\]
denote the distinct eigenvalues of the Laplace--Beltrami operator on
\(K\)\,, and
\[
P_k:L^2(K)\longrightarrow E_{\lambda_k}
\]
the corresponding spectral projections.

\mn
The following propagation theorem is Theorem~1.3 of Ganguly and
Thangavelu~\cite{GT}.

\begin{proposition}\label{quasi}
	Let \(\mathcal O\subset K\) be a nonempty open subset and let
	\(F\in L^2(K)\). Assume that $F=0$ on $\mathcal O$. Suppose that there exists a positive decreasing function
	\[
	\theta:[T,+\infty)\rightarrow(0,+\infty)\,,
	\]
	such that
	\[
	\theta(t)\to0
	\quad\text{as }t\to+\infty\,,
	\qquad
	\int_T^\infty\frac{\theta(t)}{t}\,dt=+\infty\,.
	\]
Assume moreover that
\[
\|P_kF\|_{L^2(K)}
\le
Ce^{-\sqrt{\lambda_k}\theta(\sqrt{\lambda_k})}\,,
\qquad k\ge0\,.
\]
Then
$ F\equiv0$ on $K$.
\end{proposition}

\mn 
With Proposition~\ref{quasi} at hand, the proof of
Theorem~\ref{GTthm} is identical to that of
Theorem~\ref{expthm}: one simply replaces the exponential propagation
result of Le~Rousseau--Lebeau by Proposition~\ref{quasi}. We omit the
details.

\subsection{Propagation on quasi--analytic manifolds}

We now turn to the more general setting where \(K\) is a compact
quasi--analytic Riemannian manifold. In contrast with the symmetric
space case, the proof relies on the recent work of
Bhowmik and Pradhan \cite{BP}, which extends the quasi--analytic
propagation theorem of Ganguly and Thangavelu to arbitrary compact
quasi--analytic manifolds.

\mn 
We begin by recalling the classical notion of a
Denjoy--Carleman class on a connected open set \(\Omega\subset\mathbb R^n\). Let
$
M=(M_k)_{k\ge0}
$
be an increasing logarithmically convex sequence of positive real
numbers satisfying $M_0=1$. The associated Denjoy--Carleman class \(C_M(\Omega)\) consists of all
functions \(f\in C^\infty(\Omega)\) such that, for every compact subset
\(X\subset\Omega\)\,, there exist positive constants \(C\) and \(h\)
satisfying
\begin{equation}\label{DC-estimate}
	\sup_{x\in X}
	|\partial^\alpha f(x)|
	\le
	Ch^{|\alpha|}
	M_{|\alpha|}\,,
	\qquad
	\alpha\in\mathbb N_0^n\,.
\end{equation}
The class \(C_M(\Omega)\) is said to be
\emph{quasi--analytic} if every function
\(f\in C_M(\Omega)\) whose derivatives of all orders vanish at one
point is necessarily identically zero in 
\(\Omega\). The celebrated Denjoy--Carleman theorem \cite{Rudin} asserts that
\(C_M(\Omega)\) is quasi--analytic if and only if
\begin{equation}\label{DC-condition}
	\sum_{k=1}^{\infty}
	\frac{M_{k-1}}{M_k}
	=
	+\infty\,.
\end{equation}
An important example is the real--analytic class, which is obtained by choosing $M_k=k!$.

\mn 
We assume from now on that \(M=(M_k)_{k\ge0}\) is a quasi-analytic
weight sequence associated with a regular sequence
$
m_k=\frac{M_k}{k!}
$
in the sense of Bhowmik and Pradhan~\cite{BP}. A compact Riemannian manifold \(K\) is called
\emph{quasi--analytic} if it admits an atlas whose transition maps
belong locally to the Denjoy--Carleman class \(C_M\), and whose metric
coefficients belong locally to the same class. Following \cite{BP}, we introduce the global Denjoy--Carleman class
associated with the Laplace--Beltrami operator. This definition is
intrinsic and therefore does not depend on the choice of local
coordinates:
\begin{equation}\label{CM}
	C_M(K)
	=
	\left\{
	f\in C^\infty(K):
	\|(-\Delta_K)^k f\|_{L^2(K)}
	\le
	M_{2k}\,\|f\|_{L^2(K)}\,,
	\quad k\ge0
	\right\}\,.
\end{equation}
When \(K\) is a real--analytic manifold and \(M_k=k!\), this definition
coincides with the usual class of real--analytic functions (see
\cite[Remark~1.10]{BP}). The following proposition, which is Corollary~1.13 of
Bhowmik and Pradhan~\cite{BP}, is the only property of the class
\(C_M(K)\) that we shall use.

\begin{proposition}\label{BP-cor}
	Let \(K\) be a compact connected quasi--analytic Riemannian manifold.
	If \(f\in C_M(K)\) vanishes on a measurable subset of \(K\) of positive
	measure, then \(f\equiv0\)\,.
\end{proposition}

\begin{proof}[Proof of Theorem~\ref{BPthm}]
	We adapt the proof of \cite[Theorem~1.14]{BP} to our setting. Set
	$
	A=\Lambda_g-\Lambda_{\widetilde g}\,.
	$
	Let \(\psi\in C^\infty_c(\mathcal O)\) and put
	$
	F=A\psi\,.
	$
	Since the local Dirichlet--to--Neumann maps coincide, we have
	$
	F=0
	$
	on \(\mathcal O\)\,. By the spectral decay assumption of Theorem~\ref{BPthm},
	\begin{equation}\label{weighted-spectral-sum}
		\sum_{k\ge0}
		\|P_kF\|_{L^2(K)}^2
		\exp\!\left(
		\sqrt{\lambda_k}\,
		\theta_\delta(\sqrt{\lambda_k})
		\right)
		\le C\|F\|_{L^2(K)}^2\,.
	\end{equation}
	Define
	\begin{equation}\label{weight-W}
		W(t)=
		\exp\!\left(
		\frac{t}{2\log^\delta(e+t)}
		\right)\,,
		\qquad t\ge0\,.
	\end{equation}
	Then \eqref{weighted-spectral-sum} becomes
	\begin{equation}\label{weighted-spectral-sum-W}
		\sum_{k\ge0}
		\|P_kF\|_{L^2(K)}^2
		W(\sqrt{\lambda_k})^2
		\le C\|F\|_{L^2(K)}^2\,.
	\end{equation}
	As in \cite[Theorem~1.14]{BP}, one can show that
	\[
	\sup_{t\ge0}\frac{t^k}{W(t)}
	\le
	C^k k!\,(\log (e+k))^{\delta k}\,,
	\]
	where \(C>0\) is independent of \(k\)\,. Let
	\[
	M_k=C^k k!\,(\log (e+k))^{\delta k}\,,
	\qquad k\ge0\,.
	\]
Then, by Plancherel's identity and \eqref{weighted-spectral-sum-W},
\begin{equation}\label{Lp-estimate}
	\begin{aligned}
		\|(-\Delta_K)^pF\|_{L^2(K)}^2
		&=
		\sum_{k\ge0}
		\lambda_k^{2p}
		\|P_kF\|_{L^2(K)}^2\\
		&=
		\sum_{k\ge0}
		\frac{\lambda_k^{2p}}
		{W(\sqrt{\lambda_k})^2}
		W(\sqrt{\lambda_k})^2
		\|P_kF\|_{L^2(K)}^2\\
		&\le
		\left(
		\sup_{t\ge0}
		\frac{t^{4p}}{W(t)^2}
		\right)
		\sum_{k\ge0}
		W(\sqrt{\lambda_k})^2
		\|P_kF\|_{L^2(K)}^2\\
		&\le
		C\,M_{2p}^2\,
		\|F\|_{L^2(K)}^2\,.
	\end{aligned}
\end{equation}
Consequently,
\[
\|(-\Delta_K)^pF\|_{L^2(K)}
\le
C\,M_{2p}\,
\|F\|_{L^2(K)}\,.
\]
Hence
\[
F\in C_{M'}(K)\,,
\]
where \(M'=(C\,M_{2p})_{p\ge0}\)\,.\\
By \cite[Proof of Theorem~1.14]{BP}, the sequence
\(M'=(CM_{2p})_{p\ge0}\) is again a regular quasi-analytic weight
sequence. Therefore, Corollary~1.13 of Bhowmik and Pradhan applies to
the class \(C_{M'}(K)\)\,. Since \(F\in C_{M'}(K)\) and \(F\) vanishes on
the nonempty open set \(\mathcal O\)\,, and hence on a measurable subset
of positive measure, we conclude that
$F\equiv0$ on $K$. The remainder of the proof is identical to that of
Theorem~\ref{expthm}. Using the self-adjointness of \(A\), one first
obtains \(Au=0\) on \(\mathcal O\) for every \(u\in C^\infty(K)\).
Applying the preceding argument to \(F=Au\) then yields \(Au=0\) on
\(K\)\,. Hence \(A=0\)\,, that is,
$
\Lambda_g=\Lambda_{\widetilde g}\,.
$
\end{proof}


\section{Geometric applications to warped product metrics}

In this section, we investigate the warped product model introduced in
the introduction. The purpose is to translate geometric boundary
closeness conditions on the conformal factors into the spectral
hypotheses appearing in the abstract propagation theorems of
Section~4. This is achieved by combining the diagonalization of the
Dirichlet--to--Neumann map with Weyl--Titchmarsh theory. Besides
yielding the geometric applications announced in the introduction,
this approach illustrates the bridge between geometry and
quasi-analytic propagation developed in the present paper.

\subsection{Geometric framework}

Throughout this section, we consider a class of $d$-dimensional manifolds with boundary of the form
\begin{equation}
	\label{eq:manifold}
	M=(0,1]\times K\,,
\end{equation}
where $d\geq 3$ and where \(K\) is a compact, connected and orientable
\((d-1)\)-dimensional Riemannian manifold. The manifold \(M\) has a single boundary component,
\[
\partial M=\{1\}\times K\,,
\]
which will be identified with \(K\)\,.

\medskip\noindent
We equip $M$ with a warped product metric of the form
\begin{equation}\label{Metricbis}
	g=c(r)^4\big(dr^2+r^2 g_K\big)\,,
\end{equation}
where \(g_K\) denotes the Riemannian metric on \(K\)\,, and where
\[
c:[0,1]\longrightarrow (0,\infty)
\]
is a positive function of class $C^m$, with $m\ge 2$. 

\medskip\noindent
As explained in the introduction, the geometric setting described above allows for both regular and singular Riemannian structures. 
\medskip\noindent
For the analysis of the Dirichlet--to--Neumann map, it is convenient to introduce the logarithmic variable
\[
x=-\log r,
\qquad x\in [0,+\infty)\,.
\]
In these coordinates, the manifold acquires the cylindrical representation
\begin{equation}\label{Metricx}
	g=f(x)^4\bigl(dx^2+g_K\bigr)\,,
\end{equation}
where
\[
f(x)=c(e^{-x})\,e^{-x/2}\,.
\]
Since $c$ admits a Taylor expansion of order $m$ at the origin, one obtains the asymptotic behavior
\begin{equation}\label{f}
	f(x)
	=
	e^{-x/2}
	+\sum_{k=1}^{m} c_k e^{-(k+\frac12)x}
	+o\!\left(e^{-(m+\frac12)x}\right),
	\qquad x\to+\infty\,,
\end{equation}
for suitable real coefficients $c_k$\,.

\medskip\noindent
The simplest example in this class is obtained by taking $K=S^{d-1}$ equipped with its round metric and choosing
$
c(r)\equiv 1\,.
$
In that case, $g$ coincides with the Euclidean metric on the unit ball of $\mathbb R^d$, while
$
f(x)=e^{-x/2}\,.
$
More generally, when \(K=S^{d-1}\)\,, the metrics \eqref{Metricbis}
may be viewed as conformal radial deformations of the Euclidean ball,
the deformation being entirely encoded by the conformal factor \(c\).

\subsection{The Dirichlet--to--Neumann map and the Steklov spectrum}

We now introduce the Dirichlet--to--Neumann operator associated with
the warped product manifolds introduced above.

\medskip\noindent
Given
$
\psi\in H^{1/2}(K)\,,
$
we consider the boundary value problem
\begin{equation}\label{Dirichlet}
	\left\{
	\begin{aligned}
		-\Delta_g u &=0 \qquad &&\text{ on } M\,,\\
		u&=\psi &&\text{ on } \partial M\,.
	\end{aligned}
	\right.
\end{equation}

\medskip\noindent
The Dirichlet--to--Neumann map is first defined in the weak sense by the Green identity
\[
\langle \Lambda_g\psi,\varphi\rangle
=
\int_M
\langle du,dv\rangle_g\, dV_g\,,
\]
for every $\varphi\in H^{1/2}(K)$ and every function $v\in H^1(M)$ satisfying
\[
v|_{\partial M}=\varphi\,,
\]
where $u$ is the unique solution of \eqref{Dirichlet} with boundary value $\psi$. When the metric, the boundary data and the solution are sufficiently regular, this weak definition agrees with the classical expression
\begin{equation}\label{DN}
	\Lambda_g\psi
	=
	\partial_\nu u\big|_{\partial M}\,,
\end{equation}
where $\nu$ denotes the outward unit normal vector field along $\partial M$.
Indeed, in the regular case, the endpoint \(r=0\) corresponds to a
single point \(p\) in the smooth compact completion \(\check M\)\,, so
that
$
M=\check M\setminus\{p\}.
$
Since \(\dim\check M\ge3\)\,, the point \(\{p\}\) has zero
\(H^1\)-capacity. Hence
\begin{equation}
H^1(\check M)=H^1_0(\check M\setminus\{p\})\,,
\end{equation}
where \(H^1_0(\check M\setminus\{p\})\) denotes the closure of
\(C_c^\infty(\check M\setminus\{p\})\) in the \(H^1\)-norm. Therefore,
the integration by parts formula on \(M\) follows by density from the
same formula for compactly supported smooth functions on
\(\check M\setminus\{p\}\)\,. In particular, the removed point \(p\)
produces no additional boundary contribution.

\medskip\noindent
The operator $\Lambda_g$ is self-adjoint on
$L^2(\partial M,dS_g)$ and has compact resolvent.
Its spectrum, referred to as the \emph{Steklov spectrum} of $(M,g)$,
therefore consists of a discrete sequence of real eigenvalues,
counted with multiplicities,
\begin{equation}
	\label{eq:Steklov-spectrum}
	0=\sigma_0\le \sigma_1\le \sigma_2\le \cdots,
	\qquad
	\sigma_j\to+\infty\,.
\end{equation}
We denote by
\begin{equation}
	\label{eq:Steklov-distinct}
	0=\tau_0<\tau_1<\tau_2<\cdots
\end{equation}
the distinct Steklov eigenvalues, that is, the spectrum of
$\Lambda_g$ without multiplicities. We refer to \cite{GP} for a general introduction to Steklov spectral theory.


\subsection{Separation of variables}

We next rewrite the harmonic equation \eqref{Dirichlet} in the
cylindrical variable
\[
x=-\log r\in [0,+\infty)\,.
\]
Introducing
\[
v=f^{d-2}u\,,
\]
a direct computation shows that the equation $-\Delta_g u=0$ becomes
\begin{equation}\label{LaplaceEq}
	\Bigl(-\partial_x^2-\Delta_{g_K}+q_f(x)\Bigr)v
	=
	-\frac{(d-2)^2}{4}\,v\,,
\end{equation}
where the effective potential $q_f$ is defined by
\begin{equation}\label{qf}
	q_f(x)
	=
	\frac{(f^{d-2})''(x)}{f^{d-2}(x)}
	-\frac{(d-2)^2}{4}\,.
\end{equation}
In terms of the original radial variable $r=e^{-x}$, the potential can be written as
\begin{equation}\label{potradial}
	q_f(x)
	=
	(d-2)\left(
	(d-3)r^2
	\left(\frac{c'(r)}{c(r)}\right)^2
	+
	r^2\frac{c''(r)}{c(r)}
	+
	(d-1)r\frac{c'(r)}{c(r)}
	\right)\,.
\end{equation}
Moreover, the asymptotic expansion \eqref{f} implies
\begin{equation}\label{Asympqf}
	q_f(x)=O(e^{-x}),
	\qquad x\to+\infty\,.
\end{equation}

\medskip\noindent
The product structure of $(M,g)$ allows us to separate variables.
Let
\begin{equation}
	\label{eq:laplace-spectrum-multiplicity}
	-\Delta_{g_K}Y_j=\mu_jY_j\,,
	\qquad j\ge0,
\end{equation}
where $\{Y_j\}_{j\ge0}$ is an orthonormal basis of eigenfunctions of
$L^2(K,dV_{g_K})$\,. Recall that the distinct eigenvalues of $-\Delta_{g_K}$ are denoted by
\begin{equation}
	\label{eq:laplace-spectrum-distinct}
	0=\lambda_0<\lambda_1<\lambda_2<\cdots,
\end{equation}
while
\begin{equation}
	\label{eq:laplace-spectrum-full}
	0=\mu_0\le \mu_1\le \mu_2\le\cdots
\end{equation}
denotes the spectrum counted with multiplicities. Expanding
\begin{equation}\label{SepaV}
	v(x,\omega)
	=
	\sum_{j\ge0}v_j(x)Y_j(\omega),
\end{equation}
equation \eqref{LaplaceEq} reduces to the family of one-dimensional equations
\begin{equation}\label{ODE-x}
	-v_j''+q_f(x)v_j
	=
	-\kappa_j^2v_j\,,
	\qquad x\in(0,+\infty)\,,
\end{equation}
where
\begin{equation}
	\label{eq:kappa-j}
	\kappa_j
	=
	\sqrt{\mu_j+\frac{(d-2)^2}{4}}\,.
\end{equation}
This naturally leads to the Schr\"odinger operator
\begin{equation}\label{Hoperator}
	H
	=
	-\frac{d^2}{dx^2}+q_f(x)
\end{equation}
acting on the half-line, and to the associated spectral equation
\begin{equation}\label{SLz}
	-v''+q_f(x)v=zv\,,
	\qquad z\in\mathbb C\,.
\end{equation}

\medskip\noindent
Let $\{C_0(\cdot,z),S_0(\cdot,z)\}$ be the fundamental system of solutions of \eqref{SLz} normalized by
\begin{equation}\label{FSS}
	C_0(0,z)=1\,,
	\qquad
	C_0'(0,z)=0\,,
	\qquad
	S_0(0,z)=0\,,
	\qquad
	S_0'(0,z)=1\,.
\end{equation}
Their Wronskian is therefore given by
\begin{equation}\label{Wronskian}
	W(C_0,S_0)=1\,.
\end{equation}
Since $q_f\in L^1(0,+\infty)$ by \eqref{Asympqf}, the operator $H$ is in the limit-point case at infinity. Consequently, for every
$
z\in\mathbb C,
$
there exists, up to a multiplicative constant, a unique solution
$S_\infty(\cdot,z)$ of \eqref{SLz} belonging to $L^2$ in a neighborhood of $+\infty$, (see \cite{RS3}, Theorem XI.57 where our spectral parameter $z= k^2$). Writing 
\begin{equation}\label{WeylSolution}
	S_\infty(x,z)
	=
A(z)\bigl(C_0(x,z)+M(z)S_0(x,z)\bigr)\,,
\end{equation}
defines the Weyl--Titchmarsh function $M(z)$. Using \eqref{Wronskian}, we obtain

\begin{equation}\label{WT}
	M(z)
	= 
	\frac{S_\infty'(0,z)}
	{S_\infty(0,z)}\,.
\end{equation}

\medskip\noindent
Finally, the Dirichlet--to--Neumann map is diagonal with respect to the
Hilbert basis of harmonics \(\{Y_j\}_{j\ge0}\)\,. Thus, if
$
\psi=\sum_{j\ge0}\psi_jY_j\,,
$
then
\begin{equation}
	\label{GlobalDN}
	\Lambda_g\psi
	=
	\sum_{j\ge0}
	(\Lambda_g^{(j)}\psi_j)\,Y_j\,,
\end{equation}
where \(\Lambda_g^{(j)}\) is the scalar by which \(\Lambda_g\) acts on
\(\mathrm{span}\{Y_j\}\). The computation is identical to the one carried out in
\cite[Section~2]{DKN}. If
$
\mu_j=\lambda_k\,,
$
with \(\mu_j\) counted with multiplicities and \(\lambda_k\) denoting a
distinct eigenvalue, then
\begin{equation}
	\label{DNk}
	\Lambda_g^{(j)}
	=
	\frac{(d-2)f'(0)}{f^3(0)}
	-
	\frac{M(-\rho_k^2)}{f^2(0)}\,,
\end{equation}
where
\begin{equation}
	\label{eq:rho-k}
	\rho_k
	=
	\sqrt{\lambda_k+\frac{(d-2)^2}{4}},
	\qquad k\ge0\,.
\end{equation}

\medskip\noindent
Since the coefficient in \eqref{DNk} depends only on the corresponding
distinct eigenvalue \(\lambda_k\), the operator \(\Lambda_g\) acts by scalar
multiplication on each eigenspace
\begin{equation}
	\label{eq:eigenspace-k}
	E_{\lambda_k}
	=
	\ker(-\Delta_{g_K}-\lambda_k)\,.
\end{equation}
Hence
\begin{equation}
	\label{DN-projection}
	\Lambda_g
	=
	\sum_{k\ge0}
	\left(
	\frac{(d-2)f'(0)}{f^3(0)}
	-
	\frac{M(-\rho_k^2)}{f^2(0)}
	\right)
	P_k\,.
\end{equation}
Since the Weyl--Titchmarsh function \(M\) is strictly increasing, (see Lemma~2.3 in \cite{DKN}), the distinct Steklov
eigenvalues are given by
\begin{equation}
	\label{LinkSteklovWT1}
	\tau_k
	=
	\frac{(d-2)f'(0)}{f^3(0)}
	-
	\frac{M(-\rho_k^2)}{f^2(0)}\,,
	\qquad k\ge0\,.
\end{equation}


\subsection{A weighted Laplace estimate}

The following elementary estimate is the key analytic ingredient in the
proof of the Weyl--Titchmarsh estimate established in the next
subsection. It explains the particular choice of the weight
\(\Phi\) appearing in Theorem~\ref{applicationsymetrique}.

\mn 
Let
$
\theta:[T,\infty)\longrightarrow(0,\infty)
$
be a decreasing function such that
\[
\lim_{t\to\infty}\theta(t)=0\,.
\]
Define
\begin{equation}\label{Phi-I-def}
	\Phi(x)
	=
	\inf_{t\ge T}
	\bigl(2xt-t\theta(t)\bigr)\,,
	\qquad x>0\,,
\end{equation}
and
\begin{equation}
	I(\rho)
	=
	\int_0^\varepsilon e^{-2\rho x}e^{\Phi(x)}\,dx\,,
\end{equation}
where, without loss of generality, we assume that
\(0<\varepsilon\le1\). We have the following result:

\begin{lemma}\label{LaplaceEstimate}
	For every \(x>0\), the quantity \(\Phi(x)\) is finite. Moreover,
	\[
	I(\rho)
	\le
	e^{-\rho\,\theta(\rho)}\,,
	\qquad \rho\ge T\,.
	\]
\end{lemma}

\begin{proof}
	Fix \(x>0\)\,. Since \(\theta(t)\to0\) as \(t\to+\infty\)\,, there exists
	\(T_x\ge T\) such that
	\[
	\theta(t)\le x\,,
	\qquad t\ge T_x\,.
	\]
	Hence
	\[
	2xt-t\theta(t)
	=
	t(2x-\theta(t))
	\ge xt,
	\qquad t\ge T_x\,.
	\]
	It follows that the infimum defining \(\Phi(x)\) is finite. Next, by definition of \(\Phi\)\,,
	\[
	\Phi(x)
	\le
	2\rho x-\rho\theta(\rho),
	\qquad \rho\ge T\,.
	\]
	Therefore
	\[
	e^{-2\rho x}e^{\Phi(x)}
	\le
	e^{-\rho\,\theta(\rho)}\,.
	\]
Integrating over \((0,\varepsilon)\) yields the desired estimate.
\end{proof}

\subsection{A Weyl--Titchmarsh estimate}

We now recall several estimates proved in \cite{DKN} that will be used in the
proof of our quasi-analytic propagation result. The key point is that the
difference of the Weyl--Titchmarsh functions can be represented as a Laplace
transform involving the difference of the effective potentials $q_{\widetilde f}$ and $q_f$\,.

\medskip\noindent
According to Lemma~4.4 of \cite{DKN} (with $p=1$), there exists a kernel
\[
K_1:\{(x,t)\,;\,0\le t\le x\}\longrightarrow \mathbb R
\]
such that, for every \(0<\alpha<1\)\,, there exists a constant
\(C_{\alpha}>0\) satisfying
\begin{equation}\label{K1-estimate}
	\left|
	\partial_x^k\partial_t^\ell K_1(x,t)
	\right|
	\le
	C_{\alpha}\,e^{-\alpha x}\,,
	\qquad
	0\le t\le x\,,
	\qquad
	k,\ell\le m-1\,.
\end{equation}
We then define the Volterra operator $B$ by
\begin{equation}\label{Volterra-operator}
h\mapsto	Bh(x)
	=
	h(x)
	+
	\int_0^x K_1(x,t)h(t)\,dt\,.
\end{equation}
Lemma~4.5 of \cite{DKN} asserts that, for all sufficiently large \(\rho\),
\begin{equation}\label{WT-B}
	S_\infty(0,-\rho^2)\widetilde S_\infty(0,-\rho^2)
	\bigl(
	M(-\rho^2)-\widetilde M(-\rho^2)
	\bigr)
	=
	\int_0^\infty
	e^{-2\rho x}
	B[q_{\widetilde f}-q_f](x)\,dx\,.
\end{equation}
Moreover, Corollary~4.3 of \cite{DKN} yields
\begin{equation}\label{WT-origin}
	S_\infty(0,-\rho^2)
	=
	1+O(\rho^{-1})\,,
	\qquad
	\widetilde S_\infty(0,-\rho^2)
	=
	1+O(\rho^{-1})\,,
	\qquad
	\rho\to+\infty\, .
\end{equation}
Combining \eqref{WT-B} and \eqref{WT-origin}, we obtain the
 existence of positive constants \(C\) and \(\rho_0\) such that
\begin{equation}\label{WT-B-estimate}
	\left|
	M(-\rho^2)-\widetilde M(-\rho^2)
	\right|
	\le
	C
	\left|
	\int_0^\infty
	e^{-2\rho x}
	B[q_{\widetilde f}-q_f](x)\,dx
	\right|\,,
	\mbox{ for } 
	\rho\ge \rho_0\,.
\end{equation}

\medskip\noindent
We now estimate the right-hand side of \eqref{WT-B-estimate}. We first deal
with the contribution of the interval \((\epsilon,\infty)\). Recall that
\begin{equation}\label{AsympQf}
	q_f(x)=O(e^{-x})\,,
	\qquad x\to+\infty .
\end{equation}
Since \(q_f\) and \(q_{\widetilde f}\) belong to \(C^{m-2}([0,\infty))\) and satisfy
\eqref{AsympQf}, we have
\begin{equation}\label{Hdelta}
	q_{\widetilde f}-q_f\in H_\delta\,,
	\qquad
	0<\delta<2\,,
	\qquad
	H_\delta
	=
	\Bigl\{
	q\,;\,
	\int_0^\infty |q(x)|^2e^{\delta x}\,dx<\infty
	\Bigr\}\,.
\end{equation}
Moreover, by Proposition~4.6 of \cite{DKN}, for every
\(0<\delta<1\)\,, the Volterra operator \(B\) is an isomorphism
$
B:H_\delta\longrightarrow H_\delta\,.
$
Consequently, for $0<\delta <1$,
$
B[q_{\widetilde f}-q_f]\in H_\delta \,.
$
Therefore, by the Cauchy--Schwarz inequality,
\begin{equation}\label{tail-cs}
	\left|
	\int_\varepsilon^\infty
	e^{-2\rho x}
	B[q_{\widetilde f}-q_f](x)\,dx
	\right|
	\le
	\left(
	\int_\varepsilon^\infty
	e^{-(4\rho+\delta)x}\,dx
	\right)^{1/2}
	\left\|
	B[q_{\widetilde f}-q_f]
	\right\|_{H_\delta}\,.
\end{equation}
Thus,
\begin{equation}\label{tail-estimate}
	\left|
	\int_\varepsilon^\infty
	e^{-2\rho x}
	B[q_{\widetilde f}-q_f](x)\,dx
	\right|
	\le
	C e^{-2\rho\varepsilon},
	\qquad \rho>0 \,.
\end{equation}

\medskip\noindent
We now estimate the contribution of the interval \((0,\varepsilon)\) in the right-hand
side of \eqref{WT-B-estimate}. To this end, we assume that
\begin{equation}\label{quasi-analytic-closeness-f}
	\left|
	\partial_x^j(f-\widetilde f)(x)
	\right|
	\le
	C e^{\Phi(x)}\,,
	\qquad
	0<x<\varepsilon\,,
	\qquad
	j=0,1,2.
\end{equation}
Consequently, since \(q_f\) depends smoothly on \(f\), \(f'\), and \(f''\),
and since \(f\) and \(\widetilde f\) are bounded away from zero, there exists
\(C>0\) such that
\begin{equation}\label{quasi-analytic-closeness-q}
	|q_{\widetilde f}(x)-q_f(x)|
	\le
	C e^{\Phi(x)}\,,
	\qquad
	0<x<\epsilon\,.
\end{equation}
Moreover, by \eqref{K1-estimate}, the kernel \(K_1\) is bounded on
\(\{0\le t\le x\le\varepsilon\}\). Hence, using \eqref{Volterra-operator},
\[
|B[q_{\widetilde f}-q_f](x)|
\le
|q_{\widetilde f}(x)-q_f(x)|
+
C\int_0^x
|q_{\widetilde f}(t)-q_f(t)|\,dt\,.
\]
Since \(\Phi\) is increasing, \eqref{quasi-analytic-closeness-q} yields
\[
|q_{\widetilde f}(t)-q_f(t)|
\le
C e^{\Phi(x)}\,,
\qquad 0<t<x\,,
\]
and therefore
\begin{equation}\label{Bq-local-bound}
	|B[q_{\widetilde f}-q_f](x)|
	\le
	C e^{\Phi(x)}\,,
	\qquad 0<x<\varepsilon\,.
\end{equation}
It follows that
\[
\left|
\int_0^\varepsilon
e^{-2\rho x}
B[q_{\widetilde f}-q_f](x)\,dx
\right|
\le
C
\int_0^\varepsilon e^{-2\rho x}e^{\Phi(x)}\,dx\,.
\]
Using Lemma~\ref{LaplaceEstimate}, we obtain
\begin{equation}\label{local-laplace-estimate}
	\left|
	\int_0^\varepsilon
	e^{-2\rho x}
	B[q_{\widetilde f}-q_f](x)\,dx
	\right|
	\le
	C e^{-\rho\,\theta(\rho)}\,,
	\qquad
	\rho\ge T\,.
\end{equation}
Combining \eqref{WT-B-estimate}, \eqref{tail-estimate} and
\eqref{local-laplace-estimate}, we obtain
\[
\left|
M(-\rho^2)-\widetilde M(-\rho^2)
\right|
\le
C\bigl(e^{-\rho\,\theta(\rho)}+e^{-2\rho\varepsilon}\bigr)\,.
\]
Since \(\theta(\rho)\to0\) as \(\rho\to+\infty\), it follows that
\(e^{-2\rho\varepsilon}=O(e^{-\rho\,\theta(\rho)})\). Hence there exist positive constants $C$ and $\rho_0$ such that
\begin{equation}\label{WT-final-decay}
	\left|
	M(-\rho^2)-\widetilde M(-\rho^2)
	\right|
	\le
	C e^{-\rho\,\theta(\rho)}\,,
	\mbox{ for }	\rho\geq \rho_0\,.
\end{equation}

\subsection{Proof of Theorem~\ref{applicationsymetrique}}

We now prove the geometric application announced in
Theorem~\ref{applicationsymetrique}. By
Theorem~\ref{GTthm}, it is enough to establish the required
Ingham-type spectral decay estimate for the difference of the
Dirichlet--to--Neumann maps. We first prove an elementary property of
the weight function \(\Phi\).

\begin{lemma}\label{explosion}
	Let
	\[
	\Phi(x)=\inf_{t\ge T}(2xt-t\theta(t))\,,
	\]
	where \(\theta:[T,+\infty)\to(0,+\infty)\) is a positive decreasing function satisfying
	\[
	\int_T^\infty \frac{\theta(t)}{t}\,dt=+\infty\,.
	\]
	Then
	\[
	\Phi(x)\to-\infty \text{ as } x\to0^+\,.
	\]
\end{lemma} 

\begin{proof}
	Let \(A>0\). We first observe that
	\[
	\sup_{t\ge T} t\theta(t)=+\infty\,.
	\]
	Indeed, otherwise there would exist \(C>0\) such that
	\[
	t\theta(t)\le C,
	\qquad t\ge T,
	\]
	and therefore
	\[
	\int_T^\infty \frac{\theta(t)}{t}\,dt
	\le
	C\int_T^\infty \frac{dt}{t^2}
	<+\infty\,,
	\]
	which contradicts the assumption. Hence we may choose \(t_A\ge T\) such that
	\[
	t_A\theta(t_A)\ge 2A.
	\]
	For \(x>0\) small enough so that
	\[
	2x\le \frac{\theta(t_A)}{2}\,,
	\]
	we get
	\[
	\Phi(x)
	\le
	2xt_A-t_A\theta(t_A)
	\le
	-\frac12 t_A\theta(t_A)
	\le
	-A\,.
	\]
	Since \(A>0\) is arbitrary, this proves that
	$
	\Phi(x)\to-\infty
	\qquad\text{as }x\to0^+\,.
	$
\end{proof}

\medskip\noindent
Now, assume that
\begin{equation}\label{quasi-analytic-final}
	\left|
	\partial_x^j(f-\widetilde f)(x)
	\right|
	\le
	C e^{\Phi(x)}\,,
	\qquad
	0<x<\varepsilon,
	\qquad
	j=0,1,2.
\end{equation}
By the previous lemma,
\[
e^{\Phi(x)}\longrightarrow0
\qquad\text{as }x\to0^+\,.
\]
Hence \eqref{quasi-analytic-final} implies
\begin{equation}\label{boundary-jets-final}
	f^{(j)}(0)=\widetilde f^{(j)}(0)\,,
	\qquad j=0,1,2.
\end{equation}

\mn
Let
\begin{equation}
	\label{eq:F-definition}
	F=(\Lambda_g-\Lambda_{\widetilde g})\psi\,,
	\qquad
	\psi\in C_c^\infty(\mathcal O)\,.
\end{equation}
If the local Dirichlet--to--Neumann maps coincide on \(\mathcal O\)\,,
we have
$
F=0$ on $\mathcal O$.
Combining the diagonalization formula for the Dirichlet--to--Neumann
map with the boundary identities \eqref{boundary-jets-final}, we obtain
\begin{equation}
	\label{eq:PkF-WT}
	P_kF
	=
	-\frac1{f^2(0)}
	\Bigl(
	M(-\rho_k^2)-\widetilde M(-\rho_k^2)
	\Bigr)
	P_k\psi\,.
\end{equation}
where
\[
\rho_k
=
\sqrt{\lambda_k+\frac{(d-2)^2}{4}}\,.
\]
Hence,
\begin{equation}
	\label{eq:PkF-L2}
	\|P_kF\|_{L^2(K)}
	=
	\frac1{f^2(0)}
	\bigl|
	M(-\rho_k^2)-\widetilde M(-\rho_k^2)
	\bigr|
	\,\|P_k\psi\|_{L^2(K)}\,.
\end{equation}
Using the Weyl--Titchmarsh estimate \eqref{WT-final-decay}, we deduce
\begin{equation}
\|P_kF\|_{L^2(K)}
\le
C
e^{-\rho_k\theta(\rho_k)}
\|P_k\psi\|_{L^2(K)} \le C
e^{-\rho_k\theta(\rho_k)}
\|\psi\|_{L^2(K)}\,.
\end{equation}
As
\(
\rho_k\sim\sqrt{\lambda_k},
\)
possibly replacing \(\theta\) by a smaller function with the same
properties, we obtain
\begin{equation}
\|P_kF\|_{L^2(K)}
\le
C
e^{-\sqrt{\lambda_k}\theta(\sqrt{\lambda_k})}\,.
\end{equation}
Therefore \(F\) satisfies the assumptions of
Proposition \ref{quasi}. Since \(F=0\) on \(\mathcal O\)\,, we conclude that
\(F\equiv0\) on \(K\)\,.

\mn
The remainder of the argument is identical to the final
self-adjointness argument in the proof of
Theorem~\ref{expthm}. Since \(f(0)=\widetilde f(0)\), the operator
$
A=\Lambda_g-\Lambda_{\widetilde g}
$
is self-adjoint on \(L^2(K,dS_K)\). Arguing exactly as in the proof of
Theorem~\ref{expthm}, we conclude that
\[
Au=0\,,
\qquad
u\in C^\infty(K)\,.
\]
Therefore,
$
\Lambda_g=\Lambda_{\widetilde g}\,,
$
which completes the proof.

\subsection{Proof of Theorem~\ref{applicationquasianalytic}}

The proof is identical to that of
Theorem~\ref{applicationsymetrique}, replacing
Theorem~\ref{GTthm} by Theorem~\ref{BPthm}.


\section*{Acknowledgements}
This work has received funding from the European Research Council (ERC) under the European Union's Horizon 2020 research and innovation programme through the grant agreement 862342 (A.E.). A.E. is partially supported by the grants CEX2023-001347-S, RED2022-134301-T, and PID2022-136795NB-I00 funded by the Spanish Ministry of Science and Innovation. F.N. thanks the French GDR Dynqua for his support. N.K. is supported by NSERC grant RGPIN 105490-2025.


\end{document}